\documentclass[12pt]{amsart}
\usepackage{amscd, amsfonts, amssymb, amsthm}
\usepackage{parskip, fullpage, verbatim}
\usepackage[colorlinks, breaklinks, linkcolor=blue]{hyperref} 
\usepackage{breakurl}
\usepackage{ctable}
\usepackage{longtable,aliascnt}
\usepackage{array}
\usepackage{booktabs}
\usepackage{wasysym} 
\usepackage{float}
\restylefloat{table}
\usepackage{multirow}

\newcommand\CA{{\mathcal A}} 
\newcommand\CB{{\mathcal B}}

\newcommand\BBC{{\mathbb C}}

\newcommand\BBZ{{\mathbb Z}}

\newcommand {\GAP}{\textsf{GAP}}  
\newcommand {\CHEVIE}{\textsf{CHEVIE}}  
  
\newcommand {\Sage}{\textsf{SAGE}}

\newcommand\codim{\operatorname{codim}}

\newcommand\Der{{\operatorname{Der}}}

\newcommand\GL{\operatorname{GL}}

\newcommand\pdeg{\operatorname{pdeg}}

\newcommand\Bool{\operatorname{\mathbb B}}

\newcommand\inverse{^{-1}}
\renewcommand\th{{^{\text{th}}}}

\numberwithin{equation}{section}

\theoremstyle{plain}
\newtheorem{lemma}[equation]{Lemma}
\newtheorem{theorem}[equation]{Theorem}

\newtheorem{corollary}[equation]{Corollary}
\newtheorem{proposition}[equation]{Proposition}
\theoremstyle{definition}
\newtheorem{defn}[equation]{Definition}
\newtheorem{remark}[equation]{Remark}
\newtheorem{example}[equation]{Example}

\subjclass[2010]{Primary 20F55, 52C35, 14N20; Secondary 13N15}

\begin{document}

\title[Supersolvable restrictions of reflection arrangements]
{Supersolvable restrictions of reflection arrangements}

\author[N. Amend]{Nils Amend}
\address
{Fakult\"at f\"ur Mathematik,
Ruhr-Universit\"at Bochum,
D-44780 Bochum, Germany}
\email{nils.amend@rub.de}

\author[T. Hoge]{Torsten Hoge}
\address
{Institut f\"ur Algebra, Zahlentheorie und Diskrete Mathematik,
Fakult\"at f\"ur Mathematik und Physik,
Leibniz Universit\"at Hannover,
Welfengarten 1,
30167 Hannover, Germany}
\email{hoge@math.uni-hannover.de}

\author[G. R\"ohrle]{Gerhard R\"ohrle}
\address
{Fakult\"at f\"ur Mathematik,
Ruhr-Universit\"at Bochum,
D-44780 Bochum, Germany}
\email{gerhard.roehrle@rub.de}

\keywords{Complex reflection groups,
reflection arrangements, free arrangements, 
supersolvable arrangements}

\allowdisplaybreaks

\begin{abstract}
Let $\CA = (\CA,V)$ be a complex hyperplane arrangement and let
$L(\CA)$ denote its intersection lattice.
The arrangement $\CA$ is called supersolvable, provided 
its lattice $L(\CA)$ is supersolvable.
For $X$ in $L(\CA)$, 
it is known that the restriction $\CA^X$
is supersolvable provided $\CA$ is.

Suppose that $W$ is a finite, unitary reflection group acting on the complex 
vector space $V$. Let $\CA = (\CA(W), V)$ be its associated
hyperplane arrangement.
Recently, the last two authors
classified all supersolvable reflection arrangements. 
Extending this work, 
the aim of this note is to 
determine all supersolvable restrictions 
of reflection arrangements.
It turns out that apart from the obvious 
restrictions of supersolvable reflection arrangements there are only a
few additional instances. 
Moreover, in a recent paper 
we classified all inductively free restrictions $\CA(W)^X$ of
reflection arrangements $\CA(W)$. Since every 
supersolvable arrangement is inductively free,  
the supersolvable restrictions $\CA(W)^X$ of
reflection arrangements $\CA(W)$ form a natural subclass of 
the class of  inductively free restrictions $\CA(W)^X$.

Finally, we characterize  
the irreducible supersolvable restrictions  of reflection arrangements 
by the presence of modular elements of dimension $1$ in their 
intersection lattice.
This in turn leads to the surprising fact 
that reflection arrangements as well as their
restrictions are
supersolvable if and only if they are strictly linearly fibered.
\end{abstract}

\maketitle


\section{Introduction}

Let $\CA  = (\CA,V)$ be a complex hyperplane arrangement and let
$L(\CA)$ denote its intersection lattice.
We say that $\CA$ is \emph{supersolvable}, provided $L(\CA)$
is supersolvable, see Definition \ref{def:super}.
Thanks to \cite[Prop.\ 3.2]{stanley:super} (see Corollary \ref{cor:stanley}), for $X$ in $L(\CA)$, the 
restriction $\CA^X$ of a supersolvable arrangement $\CA$ is itself
again supersolvable.

Now suppose that $W$ is a finite, unitary 
reflection group acting on the complex 
vector space $V$.
Let $\CA = (\CA(W),V)$ be the associated 
hyperplane arrangement of $W$.
In \cite[Thm.\ 1.2]{hogeroehrle:super}, we classified all 
supersolvable reflection arrangements.
Extending this earlier work, 
the aim of this note is to 
classify all supersolvable restrictions 
$\CA^X$ for $\CA$ a reflection arrangement.
Since supersolvability is a rather strong condition,
not unexpectedly, there are only 
very few additional instances
apart from the obvious 
restrictions of supersolvable reflection arrangements. 

Moreover, similar to the case of supersolvable reflection arrangements, 
we are able to characterize 
the irreducible arrangements in this class merely 
by the presence 
of modular elements of dimension $1$ in their intersection lattice
(see Theorem \ref{super-restriction-rank2}).
This in turn leads to the unexpected, remarkable fact 
that reflection arrangements as well as their
restrictions are
supersolvable if and only if they are strictly linearly fibered
(see Corollary \ref{cor:linearlyfiberd}).

The classification 
of the irreducible, finite, complex reflection groups $W$ due to 
Shephard and Todd \cite{shephardtodd}
states that each such group belongs to one of two types.
Namely, either $W$ belongs to the infinite  
three-parameter family $G(r,p,\ell)$ of monomial groups, or else is one of  
an additional 34 exceptional groups, 
simply named $G_4$ up to $G_{37}$.
As a result, proofs of properties 
of $W$ and its arrangement $\CA(W)$ frequently also do 
come in two flavors: conceptional, uniform arguments for the
infinite families on the one hand, and adhoc and  
mere computational techniques for the exceptional instances,
on the other, 
e.g.\ see \cite[\S 6, App.\ B, App.\ C]{orlikterao:arrangements}
and \cite{orliksolomon:unitaryreflectiongroups}.
This dichotomy also prevails the statements
and proofs of this paper.

First we recall the main result from \cite[Thm.\ 1.2]{hogeroehrle:super}:

\begin{theorem}
\label{super}
For $W$ a finite
complex reflection group,  
$\CA(W)$ is 
supersolvable if and only if 
any irreducible factor of $W$ is
of rank at most $2$, or 
is isomorphic either to 
a Coxeter group 
of type $A_\ell$ or $B_\ell$ for $\ell \ge 3$, or to a  
monomial group $G(r,p,\ell)$
for $r, \ell \ge 3$  and $p \ne r$.
\end{theorem}

It is easy to see that any central 
arrangement of rank at most $2$ is supersolvable, 
cf.\ \cite[Rem.\ 2.3]{hogeroehrle:super}. 
Thus we focus in the sequel on restrictions $\CA^X$ with $\dim X \ge 3$.

In \cite[Thm.\ 1.2]{amendhogeroehrle:indfree},
we classified all inductively free restrictions $\CA^X$ of
reflection arrangements $\CA$. Since a
supersolvable arrangement is inductively free, 
thanks to work of Jambu and Terao  \cite[Thm.\ 4.2]{jambuterao:free}
(see Theorem \ref{thm:superfree}),
the supersolvable restrictions $\CA^X$ of
reflection arrangements $\CA$ form a natural subclass of 
the inductively free restrictions. 

In order to state \cite[Thm.\ 1.2]{amendhogeroehrle:indfree}, we
require a  bit more notation.
For fixed $r, \ell \ge 2$ and 
$0 \le k \le \ell$, we denote by $\CA_\ell^k(r)$ the 
sequence of intermediate arrangements that interpolate
between the two reflection arrangements 
$\CA(G(r,r,\ell)) =: \CA_\ell^0(r)$ and $\CA(G(r,1,\ell)) =: \CA_\ell^\ell(r)$,
defined in \cite[\S 2]{orliksolomon:unitaryreflectiongroups}
(see also \cite[\S 6.4]{orlikterao:arrangements}).
Note that $\CA_\ell^k(r)$ is not a reflection arrangement for $k \ne 0, \ell$. 
Also note that each of the arrangements 
does arise as the restriction $\CA(W)^X$ for suitable 
$W$ and $X$, e.g.\ see \cite[Prop.\ 6.84]{orlikterao:arrangements}.
Here is \cite[Thm.\ 1.2]{amendhogeroehrle:indfree}.

\begin{theorem}
\label{thm:indfree}
Let $W$ be a finite, irreducible, complex 
reflection group with reflection arrangement 
$\CA = \CA(W)$ and let $X \in L(\CA)$. 
The restricted arrangement $\CA^X$ is inductively free 
if and only if one of the following holds:
\begin{itemize}
\item[(i)] 
$\CA$ is inductively free;
\item[(ii)] 
$W = G(r,r,\ell)$ and 
$\CA^X \cong \CA^k_p(r)$, where $p = \dim X$ and $p - 2 \leq k \leq p$; 
\item[(iii)] 
$W$ is one of $G_{24}, G_{27}, G_{29}, G_{31}, G_{33}$, or $G_{34}$ and $X \in L(\CA) \setminus \{V\}$ with  $\dim X \leq 3$.
\end{itemize}
\end{theorem}

Because the pointwise stabilizer $W_X$ of $X$ in $L(\CA)$ is itself a 
complex reflection group, by \cite[Thm.~1.5]{steinberg:differential}, 
following \cite[\S 3, App.]{orliksolomon:unitaryreflectiongroups} (cf. \cite[\S 6.4, App.\ C]{orlikterao:arrangements}),
we may  label the $W$-orbit 
of $X \in L(\CA)$ by the type $T$ say,
of $W_X$; we thus frequently denote such a restriction $\CA(W)^X$  
simply by the pair $(W,T)$.

Thanks to \cite[Prop.\ 2.5]{hogeroehrle:super}
and the factorization property of restrictions
\eqref{eq:restrproduct} below, 
the question of supersolvability of  $\CA^X$ reduces 
to the case when $\CA$ is irreducible.
Thus,  we may assume that  $W$ is irreducible.
In light of Theorem \ref{thm:indfree}, 
we can now state the main result of our paper.

\begin{theorem}
\label{super-restriction}
Let $W$ be a finite, irreducible, complex 
reflection group with reflection arrangement 
$\CA = \CA(W)$ and let $X \in L(\CA)$ with $\dim X \ge 3$.
Then the restricted arrangement 
$\CA^X$ is supersolvable
if and only if 
one of the following holds: 
\begin{itemize}
\item[(i)] $\CA$ is supersolvable;
\item[(ii)] $W = G(r,r,\ell)$ and 
$\CA^X \cong \CA_p^p(r)$ or $\CA_p^{p-1}(r)$,
where $p = \dim X$; 
\item[(iii)] $\CA^X$ is $(E_6,A_3)$, $(E_7, D_4)$, $(E_7, A_2^2)$,
or $(E_8, A_5)$.
\end{itemize}
\end{theorem}

While Theorem \ref{super-restriction} gives a complete classification 
of supersolvable restrictions 
of reflection arrangements, 
as far as isomorphism types
of such restrictions are concerned,
there is quite a bit of redundancy in its statement.
For, the last two restrictions in part (iii) of Theorem \ref{super-restriction}
are isomorphic to each other while the 
first two restrictions 
are isomorphic to ones in part (ii) 
(see Lemma \ref{lem:isoms} below)
and $\CA_p^p(r)$ is 
isomorphic to the reflection arrangment of $G(r,1,p)$,
already covered in part (i). 
Thus, 
apart from supersolvable reflection arrangements themselves,
there is only one additional family of 
supersolvable restrictions one for each dimension,
and a single exceptional case, 
up to isomorphism. 

\begin{corollary}
\label{cor:super-restriction}
Let $W$ be a finite, irreducible, complex 
reflection group with reflection arrangement 
$\CA = \CA(W)$ and let $X \in L(\CA)$ with $\dim X \ge 3$.
Then $\CA^X$ is supersolvable
if and only if  either
$\CA^X$ is isomorphic to a supersolvable reflection arrangment, 
$\CA^X \cong \CA_p^{p-1}(r)$ for some $p \ge 3$,
or $\CA^X \cong (E_7, A_2^2)$.
\end{corollary}

The definition of 
supersolvability of $\CA$ entails the existence of 
modular elements in $L(\CA)$ of any possible rank;
see \S \ref{ssect:supersolve} for the notion of 
modular elements.
In our second main result we  show that 
irreducible,
supersolvable restrictions of 
reflection arrangements 
are characterized merely by the presence
of a modular element of  dimension $1$.

\begin{theorem}
\label{super-restriction-rank2}
For $W$ a finite, irreducible
complex reflection group of rank at least $4$,  
let $\CA = \CA(W)$.
Let $X \in L(\CA)$ with $\dim X \ge 3$.
Then $\CA^X$
is supersolvable if and only if 
there exists a modular element of dimension $1$
in its lattice $L(\CA^X)$.
\end{theorem}

\begin{remark}
\label{rem:super-restriction-rank2}
(i).
Note that 
Theorem \ref{super-restriction-rank2}
also applies to the case $X = V$ and thus gives
a new characterization of 
supersolvable  
reflection arrangements, 
cf.~\cite[Thm.\ 1.3]{hogeroehrle:super}.

(ii).
The condition of irreducibility in  
Theorem \ref{super-restriction-rank2} 
is necessary, see \cite[Rem.\ 2.6]{hogeroehrle:super}.
\end{remark}

In view of \cite[Cor.~5.112, Thm.~5.113]{orlikterao:arrangements},
Theorem \ref{super-restriction-rank2} 
readily implies the following.

\begin{corollary}
\label{cor:linearlyfiberd}
For $W$ a finite, irreducible complex reflection group let $\CA = \CA(W)$ and let $X \in L(\CA)$.
Then $\CA^X$ is strictly linearly fibered  
if and only if it is of fiber type.
\end{corollary}

\begin{remark}
\label{rem:linearlyfiberd}
See \cite[Defs.~5.10, 5.11]{orlikterao:arrangements}
for the notions of \emph{strictly linearly fibered} and 
\emph{fiber type} arrangements. 
Corollary \ref{cor:linearlyfiberd}
is rather striking as the 
latter notion is considerably stronger than the first. 
In particular, it follows that for $W$ a finite, complex reflection group, 
the reflection arrangement $\CA(W)$ is supersolvable if and only if it is 
strictly linearly fibered.
This new fact underlines the very special role played by 
the class of 
reflection arrangements and their restrictions among all arrangements.
\end{remark}

The paper is organized as follows.
In \S \ref{sect:recol} we recall the required notation and  
facts about supersolvability of arrangements
and reflection arrangements from
\cite[\S 4, \S6]{orlikterao:arrangements}
and prove some preliminary results.
Theorems \ref{super-restriction} and \ref{super-restriction-rank2} are proved in \S \ref{sect:proof1} and \S \ref{sect:proof2},
respectively.
The dichotomy of the classification of the irreducible 
complex reflection groups $W$ into the three 
parameter family $G(r,p,\ell)$ and an additional 34 exceptional cases
(see \cite[Table B.1]{orlikterao:arrangements}),
mentioned above, descends to an analogous  
dichotomy of the  restrictions $\CA(W)^X$ into the three-parameter family
$\CA^k_\ell(r)$ and a small finite set of exceptional cases. 
While our proofs for the restrictions that are isomorphic to 
$\CA^k_\ell(r)$ are conceptual and uniform, the statements of 
Theorem \ref{super-restriction} and 
Corollary \ref{cor:super-restriction} indicate that there cannot 
be a uniform argument in general.
For $W$ of exceptional type, 
our arguments consist of a case-by-case analysis based
on some technical general lemmas for $3$- and $4$-arrangements
provided in Section \ref{ssect:supersolve}.
Despite the fact that the statement of 
Theorem \ref{super-restriction-rank2}
is uniform, still two instances required 
some computer calculations.

For general information about arrangements and reflection groups we refer
the reader to \cite{bourbaki:groupes}, 
\cite{orliksolomon:unitaryreflectiongroups} and 
\cite{orlikterao:arrangements}.
Throughout, we use the naming scheme
of the irreducible finite complex reflection groups due to 
Shephard and Todd, \cite{shephardtodd}.

\section{Recollections and Preliminaries}
\label{sect:recol}
\subsection{Hyperplane Arrangements}
\label{ssect:hyper}

Let $V = \BBC^\ell$ 
be an $\ell$-dimensional complex vector space.
A \emph{hyperplane arrangement} is a pair
$(\CA, V)$, where $\CA$ is a finite collection of hyperplanes in $V$.
Usually, we simply write $\CA$ in place of $(\CA, V)$.
We only consider central arrangements, i.e. the origin is contained in the center $T := \bigcap_{H \in \CA}H$ of $\CA$.
The empty arrangement in $V$ is denoted by $\Phi_\ell$.

The \emph{lattice} $L(\CA)$ of $\CA$ is the set of subspaces of $V$ of
the form $H_1\cap \dotsm \cap H_n$ where $\{ H_1, \ldots, H_n \}$ is a subset
of $\CA$. 
For $X \in L(\CA)$, we have two associated arrangements, 
firstly the subarrangement 
$\CA_X :=\{H \in \CA \mid X \subseteq H\}$
of $\CA$ and secondly, 
the \emph{restriction of $\CA$ to $X$}, defined by
$\CA^X := \{ X \cap H \mid H \in \CA \setminus \CA_X\}$.

The lattice $L(\CA)$ is a partially ordered set by reverse inclusion:
$X \le Y$ provided $Y \subseteq X$ for $X,Y \in L(\CA)$.
We have a \emph{rank} function on $L(\CA)$ defined by $r(X) := \codim_V(X)$.
The \emph{rank} $r(\CA)$ of $\CA$ is the rank of a maximal element in $L(\CA)$ with respect
to the partial order.
With this definition $L(\CA)$ is a \emph{ranked geometric lattice}, 
\cite[\S 2]{orlikterao:arrangements}.
The $\ell$-arrangement $\CA$ is called \emph{essential} provided $r(\CA) = \ell$.

Note that the restriction $\CA^X$ is also central,
so that $L(\CA^X)$ is again a geometric lattice.
Let $\CA$ be central and let
$X, Y \in L(\CA)$ with $X < Y$. 
We recall the following sublattices of $L(\CA)$ from 
\cite[Def.\ 2.10]{orlikterao:arrangements},
$L(\CA)_Y :=\{ Z \in L(\CA) \mid Z \le Y\}$, 
$L(\CA)^X :=\{ Z \in L(\CA) \mid X \le Z\}$,
and the \emph{interval}
$[X,Y] := L(\CA)_Y \cap L(\CA)^X = \{ Z \in L(\CA) \mid X \le Z \le Y\}$.

\begin{lemma}
\label{lem:interval}
Let $\CA$ be central and let
$X, Y \in L(\CA)$ with $X < Y$. 
Then $L(\CA^Y)$ is a sublattice in $L(\CA^X)$.
In particular, the former is an interval in the latter.
\end{lemma}

\begin{proof}
It follows from  \cite[Lem.\ 2.11]{orlikterao:arrangements} 
that
$L(\CA^Y) = L(\CA)^Y 
\subseteq 
L(\CA)^X = L(\CA^X)$
are both sublattices of $L(\CA)$ and that
$L(\CA^Y) = L((\CA_T)^Y) = [Y, T]$ is an interval in $ L(\CA^X)$.
\end{proof}

Let $\CA = \CA_1 \times \CA_2$ be the 
\emph{product}
of the two arrangements $\CA_1$ and $\CA_2$.
With the partial order defined on 
$L(\CA_1) \times L(\CA_2)$ by
$(X_1,X_2) \le (Y_1,Y_2)$ provided $X_1 \le Y_1$ and $X_2 \le Y_2$, 
there is a lattice isomorphism
$ L(\CA_1) \times L(\CA_2) \cong L(\CA)$ given by
$(X_1, X_2) \mapsto X_1 \oplus X_2$,
\cite[\S 2]{orlikterao:arrangements}.
It is easily seen that
for $X =  X_1 \oplus X_2 \in L(\CA)$, we have 
\begin{equation}
\label{eq:restrproduct}
\CA^X = \CA_1^{X_1} \times \CA_2^{X_2}.
\end{equation}

Note that 
$\CA \times \Phi_0 = \CA$
for any arrangement $\CA$. 
If $\CA$ is of the form $\CA = \CA_1 \times \CA_2$, where 
$\CA_i \ne \Phi_0$ for $i=1,2$, then $\CA$
is called \emph{reducible}, else $\CA$ is 
\emph{irreducible}, 
\cite[Def.\ 2.15]{orlikterao:arrangements}.

\subsection{Free Arrangements}
\label{ssect:free}

Let $S = S(V^*)$ be the symmetric algebra of the dual space $V^*$ of $V$.
If $x_1, \ldots , x_\ell$ is a basis of $V^*$, then we identify $S$ with 
the polynomial ring $\BBC[x_1, \ldots , x_\ell]$.
By denoting the $\BBC$-subspace of $S$
consisting of the homogeneous polynomials of degree $p$ (and $0$) by $S_p$,
we see that there is a natural $\BBZ$-grading
$S = \oplus_{p \in \BBZ}S_p$, where
$S_p = 0$ for $p < 0$.

Let $\Der(S)$ be the $S$-module of $\BBC$-derivations of $S$
and for $i = 1, \ldots, \ell$
define $D_i := \partial/\partial x_i$ to be the $i\th$ partial derivation of $S$.
Now $D_1, \ldots, D_\ell$ is an $S$-basis of $\Der(S)$ and
we call $\theta \in \Der(S)$
\emph{homogeneous of polynomial degree $p$}
provided 
$\theta = \sum_{i=1}^\ell f_i D_i$, 
where $f_i \in S_p$ for each $1 \leq i \leq \ell$.
In this case we write $\pdeg \theta = p$.
By defining $\Der(S)_p$ to be the $\BBC$-subspace of $\Der(S)$ consisting 
of all homogeneous derivations of polynomial degree $p$,
we see that $\Der(S)$ is a graded $S$-module:
$\Der(S) = \oplus_{p\in \BBZ} \Der(S)_p$.

Following \cite[Def.~4.4]{orlikterao:arrangements}, 
we define the $S$-submodule $D(f)$ of $\Der(S)$ for $f \in S$ by
$D(f) := \{\theta \in \Der(S) \mid \theta(f) \in f S\}$.

If $\CA$ is an arrangement in $V$,
then for every $H \in \CA$ we may fix $\alpha_H \in V^*$ with
$H = \ker(\alpha_H)$.
We call $Q(\CA) := \prod_{H \in \CA} \alpha_H \in S$
the \emph{defining polynomial} of $\CA$.

The \emph{module of $\CA$-derivations} is the $S$-submodule of $\Der(S)$ 
defined by 
\[
D(\CA) := D(Q(\CA)).
\]
The arrangement $\CA$ is said to be \emph{free} if the module of $\CA$-derivations
$D(\CA)$ is a free $S$-module.

Note that  $D(\CA)$
is a graded $S$-module $D(\CA) = \oplus_{p\in \BBZ} D(\CA)_p$,
where $D(\CA)_p = D(\CA) \cap \Der(S)_p$, see
\cite[Prop.~4.10]{orlikterao:arrangements}.
If $\CA$ is a free $\ell$-arrangement, 
then by \cite[Prop.~4.18]{orlikterao:arrangements}
the $S$-module $D(\CA)$ admits a basis of $\ell$ homogeneous derivations $\theta_1, \ldots, \theta_\ell$.
While these derivations are not unique, their polynomial 
degrees $\pdeg \theta_i$ are unique (up to ordering).
The set of \emph{exponents} of the free arrangement $\CA$ is the multiset
\[
\exp\CA := \{\pdeg \theta_1, \ldots, \pdeg \theta_\ell\}.
\]
For the stronger notion of  \emph{inductive freeness}, see
\cite[Def.~4.53]{orlikterao:arrangements}.

\subsection{Supersolvable Arrangements}
\label{ssect:supersolve}

Let $\CA$ be an arrangement.
Following \cite[\S 2]{orlikterao:arrangements}, we say
that $X \in L(\CA)$ is \emph{modular}
provided $X + Y \in L(\CA)$ for every $Y \in L(\CA)$.
Let $\CA$ be a central (and essential) $\ell$-arrangement.
The next notion is due to Stanley \cite{stanley:super}. 

\begin{defn}
\label{def:super}
We say that $\CA$ is 
\emph{supersolvable} 
provided there is a maximal chain
\[
V = X_0 < X_1 < \cdots < X_{\ell-1} < X_\ell = T
\]
 of modular elements $X_i$ in $L(\CA)$,
cf.\ \cite[Def.\ 2.32]{orlikterao:arrangements}.
\end{defn}

It is easy to see that 
every $2$-arrangement is supersolvable,
\cite[Rem.\ 2.3]{hogeroehrle:super}.
The next result is also straightforward, 
\cite[Lem.\ 2.4]{hogeroehrle:super}.

\begin{lemma}
\label{lem:super-rank2}
A $3$-arrangement $\CA$ is supersolvable if and only if 
there exists a modular element in $L(\CA)$ of
dimension $1$.
\end{lemma}

Thanks to \cite[Prop.\ 2.5]{hogeroehrle:super},
supersolvable  arrangements behave well with respect to 
the  product construction for arrangements from \S \ref{ssect:hyper}.

Supersolvable arrangements are always free,
\cite[Thm.~4.58]{orlikterao:arrangements}:

\begin{theorem}
\label{thm:superfree}
Let $\CA$ be a supersolvable $\ell$-arrangment with  
maximal chain 
\[
V = X_0 < X_1 < \cdots < X_{\ell-1} < X_\ell = T
\]
of modular elements $X_i$ in $L(\CA)$. 
Define $b_i = |\CA_{X_i}\setminus \CA_{X_{i-1}}|$ for $1 \le i \le \ell$.
Then $\CA$ is inductively free with 
$\exp \CA = \{b_1, \ldots, b_\ell\}$.
\end{theorem}

Note that it was first proven in \cite[Thm.~4.2]{jambuterao:free} that every supersolvable arrangement is inductively free.
We proceed with some further preliminary results.

\begin{lemma}
\label{lem:modularrestrictions}
Let $X \in L(\CA)$ be modular and $Z \in L(\CA)$. Then $X\cap Z$ is
modular in $L(\CA^Z)$.
\end{lemma}

\begin{proof}
Let $Y \in L(\CA^Z) = L(\CA)^Z \subseteq L(\CA)$. 
Since $X$ is modular in $L(\CA)$, we 
have $X+Y \in L(\CA)$ and a direct calculation shows 
$(Z \cap X) + Y = Z \cap (X + Y)\in L(\CA^Z)$. 
\end{proof}

The following 
immediate consequence of 
Lemma \ref{lem:modularrestrictions}
is due to 
Stanley, \cite[Prop.\ 3.2]{stanley:super}.

\begin{corollary}
\label{cor:stanley}
Let $X < Y$ in $L(\CA)$.
If $\CA$ is supersolvable, then so is the
interval $[X,Y]$. 
\end{corollary}

Here is a further consequence of 
Lemma \ref{lem:modularrestrictions}. 

\begin{corollary}
\label{cor:dim4notsupbycounting}
Let $\CA$ be a $4$-arrangement and let $Z \in L(\CA)$ be of 
dimension $1$.
Suppose $\vert \CA_Z \vert > \vert \{H \in \CA \mid \CA^H 
\text{ is supersolvable } \} \vert $. Then $Z$ is not modular in $L(\CA)$.
\end{corollary}

\begin{proof}
By assumption on $\CA_Z$, there is a $H \in \CA$ with $Z \in L(\CA^H)$ such that $\CA^H$ is not 
supersolvable. Therefore, since $\CA^H$ is
a $3$-arrangement, $Z$ is not modular in $L(\CA^H)$, 
by Lemma \ref{lem:super-rank2}. 
We conclude that $Z$ is not modular in $L(\CA)$,
thanks to Lemma \ref{lem:modularrestrictions}.
\end{proof}

\begin{lemma}
\label{lem:countres}
Let $\CA$ be an arrangement. Then for $X \in L(\CA)$, we have 
$\vert \CA_X \vert > \vert (\CA^H)_X \vert$ for all $H \in \CA_X$.
\end{lemma}

\begin{proof}
Note that
$\vert(\CA^H)_X\vert 
= \vert \{H \cap H' \mid H' \in \CA_X \setminus \{H\}\}\vert < \vert \CA_X\vert$.
\end{proof}

The next result gives a useful numerical criterion for the non-supersolvability of a free,
irreducible $3$-arrangement $\CA$ in terms of the exponents of $\CA$.

\begin{lemma}
\label{lem:dim3notsuper}
Let $\CA$ be a free, irreducible $3$-arrangement.
Let $\exp \CA = \{1, b_1, b_2\}$ with $b_1 \le b_2$. 
If $|\CA_X|\le b_1$ for all $X \in L(\CA)$ of 
dimension $1$, 
then $\CA$ is not supersolvable.
\end{lemma}

\begin{proof}
Note that since $\CA$ is irreducible, $b_1 > 1$, 
\cite[Thm.~4.29(3)]{orlikterao:arrangements}.
Assume $|\CA_X|\le b_1$ and suppose that  
$\CA$ is supersolvable with $V < H < X < 0$ a maximal chain of 
modular elements in $L(\CA)$. Observe that 
$|\CA_H \setminus \CA_V| = |\{H\} \setminus \emptyset| = 1$.
By Theorem \ref{thm:superfree}, we have 
$|\CA_X \setminus \CA_H| \in \{b_1,b_2\}$. However, 
$|\CA_X \setminus \CA_H| = |\CA_X| -1 < b_1 \le b_2$, 
a contradiction.
\end{proof}

The following technical lemma is our key tool in order to show that there are 
no modular elements of dimension $1$ in certain non-supersolvable, 
irreducible $4$-arrangements.

\begin{lemma}
\label{lem:show4nonmodular1}
Let $\CA$ be a free, irreducible $4$-arrangement.
Assume there is, up to lattice isomorphism, only one supersolvable restriction 
$\CB:= \CA^H$ to
a hyperplane $H \in \CA$. 
Let $1 < b_1 \le b_2$ be the exponents of $\CB$ and let
$c = \vert\{ H' \in \CA \mid L(\CA^{H'}) \cong L(\CB)\} \vert$.
If $c \le b_1 + 1$, then there are no modular 
elements of dimension $1$ in $L(\CA)$.
\end{lemma}

\begin{proof}
Suppose that $X \in L(\CA)$ is modular of dimension $1$. Then $X$
is also modular in $\CA^H$ for all $H \in \CA_X$,
by Lemma \ref{lem:modularrestrictions},  and since $\CA^H$ is a 
$3$-arrangement, it is supersolvable, by Lemma \ref{lem:super-rank2}.
Lemma \ref{lem:dim3notsuper} implies that $|(\CA^H)_X| > b_1$ and Lemma 
\ref{lem:countres} shows that $|\CA_X| > b_1 + 1 \ge c$. This
is a contradiction to Corollary  \ref{cor:dim4notsupbycounting}.
\end{proof}

\subsection{Reflection Groups and Reflection Arrangements}
\label{ssect:refl}
The irreducible finite complex reflection groups were 
classified by Shephard and Todd, \cite{shephardtodd}.
Let $W  \subseteq \GL(V)$ be a finite complex reflection group.

The \emph{reflection arrangement} $\CA = \CA(W)$ of $W$ in $V$ is 
the hyperplane arrangement 
consisting of the reflecting hyperplanes of the elements in $W$
acting as reflections on $V$.

The intermediate arrangements $\CA^k_\ell(r)$ (with $0 \leq k \leq \ell$) interpolate between the reflection
arrangements of $G(r,r,\ell)$ and $G(r,1,\ell)$.
Let $H_i := \ker(x_i)$ and $H_{ij}(\zeta) := \ker(x_i - \zeta x_j)$, then $\CA^k_\ell(r) = \{H_1, \ldots, H_k, H_{ij}(\zeta) \mid
1 \leq i < j \leq \ell, \zeta^r = 1\}$. Note that $H_{ij}(\zeta) = H_{ji}(\zeta\inverse)$. We now recall
\cite[Prop.~2.11]{orliksolomon:unitaryreflectiongroups} (see also \cite[Prop.~6.82]{orlikterao:arrangements}):

\begin{proposition}
\label{prop:intermediate}
Let $\CA = \CA^k_\ell(r)$ and let $H \in \CA$. The type of $\CA^H$ is given in Table \ref{intermediate-table}.
\end{proposition}

\begin{table}[ht!b]
\renewcommand{\arraystretch}{1.5}
\begin{tabular}{llll}\hline
 $k$ & \multicolumn{2}{l}{$H$} & Type of $\CA^H$\\ \hline
 $0$ & arbitrary & & $\CA^1_{\ell - 1}(r)$\\
 $1, \ldots, \ell - 1$ & $H_{ij}(\zeta)$ & $1 \leq i < j \leq k < \ell$ & $\CA^{k - 1}_{\ell - 1}(r)$\\
 $1, \ldots, \ell - 1$ & $H_{ij}(\zeta)$ & $1 \leq i \leq k < j < \ell$ & $\CA^k_{\ell - 1}(r)$\\
 $1, \ldots, \ell - 1$ & $H_{ij}(\zeta)$ & $1 < k < i < j \leq \ell$ & $\CA^{k + 1}_{\ell - 1}(r)$\\
 $1, \ldots, \ell - 1$ & $H_i$ & $1 \leq i \leq \ell$ & $\CA^{\ell - 1}_{\ell - 1}(r)$\\
 $\ell$ & arbitrary & & $\CA^{\ell - 1}_{\ell - 1}(r)$\\ \hline
\end{tabular}
\medskip
\caption{Restriction types of $\CA^k_\ell(r)$}
\label{intermediate-table}
\end{table}

Let $\CA = \CA^k_\ell(r)$, fix $H \in \CA$
and let $(\CA, \CA' = \CA\setminus\{H\}, \CA'' = \CA^H)$ be the triple with 
respect to $H$.
In Table \ref{restriction-table}, we consider the map
$\CA' \to \CA''$ given by $H' \mapsto H' \cap H$, 
where restricting to $H_i$ results in the substitution $x_i \to 0$, and 
restricting to $H_{ij}(\zeta)$ results in the substitution $x_i \to \zeta x_j$.

\begin{table}[ht!b]
\renewcommand{\arraystretch}{1.5}
\begin{tabular}{lll}\hline
$H'$ & restricts to & $\tilde{H} \in \CA^{H_n}$\\ \hline\hline
$H_i$ & $i \neq n$ & $H_i$\\ \hline
$H_{ij}(\zeta)$ & $i, j \neq n$ & $H_{ij}(\zeta)$\\
& $i = n$ & $H_j$\\
& $j = n$ & $H_i$\\ \hline
& \\
& \\
\end{tabular}
\qquad
\begin{tabular}{lll}\hline
$H'$ & restricts to & $\tilde{H} \in \CA^{H_{mn}(\eta)}$\\ \hline\hline
$H_i$ & $i \neq m$ & $H_i$\\
& $i = m$ & $H_n$\\ \hline
$H_{ij}(\zeta)$ & $i, j \neq m$ & $H_{ij}(\zeta)$\\
& $i = m, j \neq n$ & $H_{nj}(\eta\inverse\zeta)$\\
& $i \neq n, j = m$ & $H_{in}(\zeta\eta)$\\ \hline
$H_{mn}(\zeta)$ & $\zeta \neq \eta$ & $H_n$\\ \hline
\end{tabular}
\medskip
\caption{Restrictions of $\CA = \CA^k_\ell(r)$ to hyperplanes $H_n$ and $H_{mn}(\eta)$}
\label{restriction-table}
\end{table}

\section{Proof of Theorem \ref{super-restriction}}
\label{sect:proof1}

We maintain the notation from the previous sections.
Note again that $\CA_\ell^k(r)$ is not a reflection arrangement for $k \ne 0, \ell$. 

\begin{lemma}
\label{lem:intermediates}
For $r \geq 2$, $\ell \ge 3$ and  $0 \le k \le \ell$,
the arrangement $\CA_\ell^k(r)$
is supersolvable if and only if $\ell - 1 \le k \le \ell$ or $(k,\ell,r) = (0,3,2)$.
\end{lemma}

\begin{proof}
First note that $\CA_3^0(2)$ is just the Coxeter arrangement of type $D_3 = A_3$.
Thus $\CA_3^0(2)$ is supersolvable, by Theorem \ref{super}.

Else for $k = 0$ and $k= \ell$, the result follows from 
Theorem \ref{super} and for 
$1 \le k \le \ell-3$, the result is a consequence of
Theorems \ref{thm:indfree}(ii)
and \ref{thm:superfree}.
For $k = \ell-1$, this follows from the case for $k = \ell$, 
the proof of \cite[Thm.~1.2(iii)]{hogeroehrle:super} and 
\cite[Lem.\ 2.62]{orlikterao:arrangements}.
Finally, for $k =\ell -2$, we argue by induction on $\ell$ as follows.
If $\CA$ is any (central) arrangement and $H$ is in $\CA$, then $L(\CA^H)$ is an intervall
in $L(\CA)$. By Corollary \ref{cor:stanley},  if $\CA^H$ is not supersolvable, then neither is $\CA$.
Suppose $\CA = \CA^{\ell - 2}_\ell(r)$ and $H = \ker(x_i - \zeta x_j) \in \CA^{\ell - 2}_\ell(r)$ for some $1 \leq i < j \leq \ell - 2$ and
some $r\th$ root of unity $\zeta$. By Proposition \ref{prop:intermediate}, we have $\CA^H \cong \CA^{\ell - 3}_{\ell - 1}(r)$,
so it suffices to show that $\CA^1_3(r)$ is not supersolvable.
So, let $\CA = \CA^1_3(r)$ and let 
$X \in L(\CA)$ be of rank $2$. Then  $\vert\CA_X\vert \in \{2, 3, r, r + 1\}$.
Since $\exp\CA = \{1, r + 1, 2r - 1\}$,  Lemma \ref{lem:dim3notsuper} implies that
$\CA^1_3(r)$ is not supersolvable.
\end{proof}

Note that it was already observed in \cite[Ex.\ 5.5]{jambuterao:free}
that $\CA_3^1(2)$ is not supersolvable.

The following example shows that the supersolvable 
arrangements from Lemma \ref{lem:intermediates}
do actually occur as restrictions of the reflection 
arrangement of $W = G(r,r,\ell)$.

\begin{example}
Let $ p,r \ge 2$, $\ell = 2p-1$, and 
$\zeta$ is an $r$th root of unity.
Let $W = G(r,r,\ell)$ and 
for $1 \le i \ne j \le \ell$ 
let $H_{i,j}(\zeta) = \ker(x_i - \zeta x_j)$ be a hyperplane in 
$\CA = \CA(W)$.
Then $X = \cap_{i=1}^{p-1} H_{2i-1,2i}(\zeta)$
belongs to $L(\CA)$ with $\dim X = \ell - (p-1) = p$. 
According to 
the description of the orbit structure of $W$ on $L(\CA)$ in
\cite[Prop.~2.14]{orliksolomon:unitaryreflectiongroups}
(cf. \cite[Prop.\ 6.84]{orlikterao:arrangements}), we have
$\CA^X \cong \CA_p^{p-1}(r)$
which is supersolvable, by Lemma \ref{lem:intermediates}.

If $Y$ is in $L(\CA)$ with $\dim Y = p$ such that $Y$ is contained in
at least one coordinate hyperplane $\ker(x_i)$, then again by 
\cite[Prop.~2.14]{orliksolomon:unitaryreflectiongroups}
(\cite[Prop.\ 6.84]{orlikterao:arrangements}),
we see that
$\CA^Y \cong \CA_p^p(r) = \CA(G(r,1,p))$
which is supersolvable, by Theorem \ref{super}.
\end{example}

As indicated in the Introduction, 
we use the convention 
to label the $W$-orbit 
of $X \in L(\CA)$ by the type $ T$ which is 
the Shephard-Todd label \cite{shephardtodd}
of the complex reflection group $W_X$.
We then denote the restriction $\CA^X$ simply by the pair
$(W,T)$.
The following result is due to Orlik and Terao, 
\cite[App.\ D]{orlikterao:arrangements}.

\begin{lemma}
\label{lem:isoms}
We have the following lattice isomorphisms of 
$3$-dimensional restrictions:
\begin{itemize}
\item [(i)] $(E_6,A_3) \cong \CA_3^2(2)$;
\item [(ii)] $(E_7, D_4) \cong \CA_3^3(2)$;
\item [(iii)] $(F_4,\tilde{A_1}) \cong (F_4,A_1) \cong (E_7,A_1^4) \cong (E_7,(A_1A_3)') \cong (E_8,A_1 D_4) \cong (E_8,D_5)$;
\item [(iv)] $(E_6,A_1 A_2) \cong (E_7,A_4)$;
\item [(v)] $(E_7,A_2^2) \cong (E_8,A_5)$;
\item [(vi)] $(E_8,A_1^2 A_3) \cong (E_8,A_2 A_3)$;
\item [(vii)] $(G_{26},A_0) \cong (G_{32}, C(3)) \cong (G_{34}, G(3,3,3))$. 
\end{itemize}
\end{lemma}

Armed with Lemma 
\ref{lem:dim3notsuper}, 
we can now determine all $3$-dimensional 
supersolvable restrictions for an ambient   
irreducible, non-supersolvable  
reflection arrangement.

\begin{lemma}
\label{lem:dim3}
Let $\CA = \CA(W)$ be an irreducible, non-supersolvable  
reflection arrangement of exceptional type.
Let $X \in L(\CA)$ with $\dim X = 3$.
Then $\CA^X$ is supersolvable if and only if 
$\CA^X$ is 
$(E_6,A_3)$, $(E_7, D_4)$, $(E_7, A_2^2)$,
or $(E_8, A_5)$. 
\end{lemma}

\begin{proof}
Using Theorem \ref{super}, the tables of all orbit types 
for the irreducible reflection groups of exceptional type in 
\cite[App.]{orliksolomon:unitaryreflectiongroups} (see also \cite[App.\ C]{orlikterao:arrangements})
and some 
explicit computer aided calculations (cf. Remark \ref{rem:computations}),
we check which of the $3$-dimensional 
restrictions satisfies the hypothesis of Lemma 
\ref{lem:dim3notsuper}.
In Table \ref{table:rank3restrictions}
we present in each of the 3-dimensional restrictions $\CB = \CA(W)^X$
of the exceptional cases all values
$|\CB_Y|$, where $Y \in L(\CB)$ is of dimension $1$
(disregarding the lattice isomorphisms from Lemma \ref{lem:isoms}).
Each of the cases labelled ``false'' satisfies the condition from 
Lemma  \ref{lem:dim3notsuper} and is therefore not 
supersolvable.

\begin{longtable}{l|l|c|c}  
$\CB :=\CA^X$ & $\exp \CB$ & $\{|\CB_{Y}|\mid \operatorname{rk}(Y) = 2\}$    
 & supersolvable\\\endfirsthead
\endhead\toprule 
\label{table:rank3restrictions}
$(F_4,A_1)$        &  1,5,7        & \{2,3,4\}        &false\\
$(F_4,\tilde A_1)$ &  1,5,7        & \{2,3,4\}        &false\\
$(G_{29},A_1)$     &  1,9,11       & \{2,3,4,5\}      &false\\
$(H_4,A_1)$        &  1,11,19      & \{2,3,5,6\}      &false\\
$(G_{31},A_1)$     &  1,13,17      & \{2,3,6\}        &false\\
$(G_{32},C(3))$    &  1,7,13       & \{2,4,5\}        &false\\
$(G_{33},A_1^2)$   &  1,7,9        & \{2,3,4,5\}      &false\\
$(G_{33},A_2)$     &  1,6,7        & \{2,3,4,5\}      &false\\
$(G_{34},A_1^3)$   &  1,13,19      & \{2,3,4,5,8\}    &false\\
$(G_{34},A_1A_2)$  &  1,13,16      & \{2,3,4,5,7\}    &false\\
$(G_{34},A_3)$     &  1,11,13      & \{2,3,4,6\}      &false\\
$(G_{34},G(3,3,3))$ &  1,7,13       & \{2,4,5\}        &false\\
$(E_6,A_1^3)$      &  1,4,5        & \{2,3,4\}        &false\\
$(E_6,A_1A_2)$     &  1,4,5        & \{2,3,4\}        &false\\
$(E_6,A_3)$        &  1,3,4        & \{2,3,4\}        & {\bf true}\\
$(E_7,A_1^4)$      &  1,5,7        & \{2,3,4\}        &false\\
$(E_7,A_1^2A_2)$   &  1,5,7        & \{2,3,4,5\}      &false\\
$(E_7,A_2^2)$      &  1,5,7        & \{2,3,4,6\}      & {\bf true}\\
$(E_7,(A_1A_3)')$  &  1,5,7        & \{2,3,4\}        &false\\
$(E_7,(A_1A_3)'')$ &  1,5,5        & \{2,3,4\}        &false\\
$(E_7,A_4)$        &  1,4,5        & \{2,3,4\}        &false\\
$(E_7,D_4)$        &  1,3,5        & \{2,3,4\}        & {\bf true}\\
$(E_8,A_1^3A_2)$   &  1,7,11       & \{2,3,4,5,6\}    &false\\
$(E_8,A_1A_2^2)$   &  1,7,11       & \{2,3,4,6\}      &false\\
$(E_8,A_1^2A_3)$   &  1,7,9        & \{2,3,4,6\}      &false\\
$(E_8,A_2A_3)$     &  1,7,9        & \{2,3,4,6\}      &false\\
$(E_8,A_1A_4)$     &  1,7,8        & \{2,3,4,5\}      &false\\
$(E_8,A_5)$        &  1,5,7        & \{2,3,4,6\}      & {\bf true}\\
$(E_8,A_1D_4)$     &  1,5,7        & \{2,3,4\}        &false\\
$(E_8,D_5)$        &  1,5,7        & \{2,3,4\}        &false\\
\hline
\caption{The supersolvable 3-dimensional restrictions of the exceptional groups}
\end{longtable}

Next we argue that the 4 cases labelled ``true'' in 
Table \ref{table:rank3restrictions} are indeed supersolvable:

The restrictions
$(E_6,A_3)$ and $(E_7,D_4)$ are supersolvable, thanks to Lemmas 
\ref{lem:intermediates}  and \ref{lem:isoms}(i),~(ii).
Moreover, by Lemma \ref{lem:isoms}(v), $(E_7,A_2^2) \cong (E_8,A_5)$, so we only have 
to check the case $(E_7,A_2^2)$.

So, let $\CB = \{H_0,\ldots,H_{12}\}$ be the arrangement given by $(E_7,A_2^2)$. 
Since $\CB$ is a $3$-arrangement, we only need to compute the
intersections of rank $2$. 
Each $X \in L(\CB)$ is uniquely given by $\CB_X$. The $X \in L(\CB)$ of rank
$2$ correspond to the following subarrangements $\CB_X$ of $\CB$:

\begin{table}[h!t]
\begin{tabular}{rrrr}
$\{H_0, H_1, H_3, H_6, H_9, H_{12}\}$, & $\{H_0, H_2, H_{10}\}$,    & $\{H_0, H_4, H_5, H_{11}\}$, & $\{H_0, H_7, H_8\}$, \\
$\{H_1, H_2, H_7, H_{11}\}$,           & $\{H_1, H_4, H_{10}\}$,    & $\{H_1, H_5, H_8\}$,         &                      \\
$\{H_2, H_3, H_4\}$,                   & $\{H_2, H_5, H_6\}$,       & $\{H_2, H_8, H_{12}\}$,      & $\{H_2, H_9\}$,      \\
$\{H_3, H_5, H_7\}$,                   & $\{H_3, H_8\}$,            & $\{H_3, H_{10}\}$,           & $\{H_3, H_{11}\}$,   \\ 
$\{H_4, H_6, H_7\}$,                   & $\{H_4, H_8, H_9\}$,       & $\{H_4, H_{12}\}$,           &                      \\
$\{H_5, H_9, H_{10}\}$,                & $\{H_5, H_{12}\}$,         &                              &                      \\ 
$\{H_6, H_8, H_{10}, H_{11}\}$,        &                            &                              &                      \\
$\{H_7, H_9\}$,                        & $\{H_7, H_{10}, H_{12}\}$, &                              &                      \\
$\{H_9, H_{11}\}$,                     &                            &                              &                      \\ 
$\{H_{11}, H_{12}\}$.                  &                            &                              &                      \\            
\end{tabular}
\end{table}

Let $Z:=H_0 \cap H_1 \cap H_3 \cap H_6 \cap H_9 \cap H_{12}$. 
For all $X \in L(\CB)$ of dimension $1$ we have $X + Z \in L(\CB)$, since $X$ and 
$Z$ are subsets of a common hyperplane of $\CB$. Therefore, $Z$ is a modular
element of dimension $1$ in $L(\CB)$ and thus $\CB$ is supersolvable, by Lemma \ref{lem:super-rank2}.
\end{proof}

Now, with the aid of Corollary \ref{cor:stanley} and 
Lemma \ref{lem:dim3} we can show that among 
irreducible, non-supersolvable  
reflection arrangement of exceptional type there are no 
$4$-dimensional supersolvable restrictions.

\begin{lemma}
\label{lem:dim4s}
Let $\CA = \CA(W)$ be an irreducible, non-supersolvable  
reflection arrangement of exceptional type.
Let $X \in L(\CA)$ with $\dim X = 4$.
Then $\CA^X$ is not supersolvable.
\end{lemma}

\begin{proof}
Thanks to Lemma \ref{lem:dim3}, for every $W$ as in the statement, other than 
$W$ of type $E_6$, $E_7$ or $E_8$, already 
every $3$-dimensional restriction is not supersolvable.
It thus follows from Corollary \ref{cor:stanley} that 
in these cases 
no higher-dimensional 
restriction is supersolvable either.

So we are left to check the following 
restrictions $(E_6,A_1^2)$, $(E_6,A_2)$, $(E_7,(A_1^3)')$,
$(E_7,(A_1^3)'')$, $(E_7,A_1A_2)$,$(E_7,A_3)$, $(E_8,A_1^4)$, $(E_8,A_1^2A_2)$,
$(E_8,A_2^2)$, $(E_8,A_1A_3)$, $(E_8,A_4)$, $(E_8,D_4)$.
Again by Corollary \ref{cor:stanley}, it suffices  
to find a restriction to a hyperplane in these cases which is 
not supersolvable; in the following table, we exhibit 
a suitable restriction in each instance: 
\begin{center}
\begin{tabular}{l|l}  
  $\CA^X$ & non-supersolvable restriction \\
\hline
$(E_6,A_1^2)$     &$(E_6,A_1A_2)$\\
$(E_6,A_2)$       &$(E_6,A_1A_2)$\\
$(E_7,(A_1^3)')$  &$(E_7,A_1^4)$ \\
$(E_7,(A_1^3)'')$ &$(E_7,A_1^4)$\\
$(E_7,A_1A_2)$    &$(E_7,A_4)$\\
$(E_7,A_3)$       &$(E_7,A_4)$\\
$(E_8,A_1^4)$     &$(E_8,A_1D_4)$\\
$(E_8,A_1^2A_2)$   &$(E_8,A_1A_4)$\\
$(E_8,A_2^2)$     &$(E_8,A_2A_3)$ \\
$(E_8,A_1A_3)$    &$(E_8,A_2A_3)$\\
$(E_8,A_4)$       &$(E_8,A_1A_4)$\\
$(E_8,D_4)$       &$(E_8,A_1D_4)$\\
\end{tabular}
\end{center}
This is readily extracted from the information in the tables of
\cite[App.]{orliksolomon:unitaryreflectiongroups} (\cite[App.\ C]{orlikterao:arrangements}).
\end{proof}

\begin{corollary}
\label{cor:dimge4}
Let $\CA = \CA(W)$ be an irreducible, non-supersolvable  
reflection arrangement of exceptional type.
Let $X \in L(\CA)$ with $\dim X \ge 4$.
Then $\CA^X$ is not supersolvable.
\end{corollary}

\begin{proof}
Thanks to Lemma \ref{lem:dim4s},  
every restriction of dimension $4$ is non-supersolvable.
The result then follows from Corollary \ref{cor:stanley}.
\end{proof}

\begin{proof}[Proof of Theorem \ref{super-restriction}]
Thanks to Corollary \ref{cor:stanley},
every interval of a supersolvable lattice is itself supersolvable.
Consequently, if $\CA$ is supersolvable, then so is any restriction $\CA^X$, 
by Lemma \ref{lem:interval}.
If $\CA^X$ is as in (ii) or (iii), then 
$\CA^X$ is supersolvable, thanks to 
Lemmas \ref{lem:intermediates} and \ref{lem:dim3} and 
Corollary \ref{cor:dimge4}.
This gives the reverse implication.

If $\CA^X$ is supersolvable but  $\CA$ is not,
then $\CA^X$  is as in (ii) or (iii), again 
by Lemmas \ref{lem:intermediates} and \ref{lem:dim3} and 
Corollary \ref{cor:dimge4}.
This gives the forward implication.
\end{proof}

\section{Proof of Theorem \ref{super-restriction-rank2}}
\label{sect:proof2}

We start by determining all modular elements in the lattice of $\CA^k_\ell(r)$.  We introduce the following notation:
Let $\Bool_k$ be the Boolean arrangement in $\BBC^k$ and define $\CB^k_\ell := \Bool_k \times \Phi_{\ell - k}$ to be the
$\ell$-arrangement consisting of the first $k$ coordinate hyperplanes in $\BBC^\ell$, i.e. $\CB^k_\ell = (\{H_1, \ldots, H_k\},
\BBC^\ell)$. Then clearly $\CB^k_\ell \subseteq \CA^m_\ell(r)$ for all $m \geq k$.

\begin{lemma}
\label{lem:modularelements}
Let $\CA = \CA^k_\ell(r)$ with $\ell \geq 3$, $0 \leq k \leq \ell$ and $r \geq 2$. 
Then all elements of $L(\CB^k_\ell) \subseteq
L(\CA)$ are modular in $L(\CA)$.

\begin{proof}
Let $X \in L(\CB^k_\ell)$ with $1 \leq \dim X < \ell - 1$ (i.e.\ $X$ is not trivially modular) and $I = \{1, \ldots ,\ell\}$.
Then there is a subset $\Delta_X \subseteq \{1, \ldots, k\}$ such that $X = \bigcap\limits_{i \in \Delta_X}H_i$.

Let $Y \in L(\CA)$ with $\dim Y = p$. Then using the construction in \cite[\S 2]{orliksolomon:unitaryreflectiongroups}
(cf. \cite[\S 6.4]{orlikterao:arrangements}), we can find a subset $\Delta_Y \subseteq I$ and a partition $\Lambda_Y =
(\Lambda_1, \ldots, \Lambda_p)$ of $I\setminus\Delta_Y$ together with $r\th$ roots of unity $\theta_{ij}$ such that
$$Y = \bigcap\limits_{i \in \Delta_Y}H_i \cap \bigcap\limits_{i, j \in \Lambda_1}H_{ij}(\theta_{ij}) \cap \cdots
\cap \bigcap\limits_{i, j \in \Lambda_p}H_{ij}(\theta_{ij}),$$
where we agree that $H_{ii}(\theta) = \BBC^\ell$ for all $i$ and all $\theta$. 
Consequently, in the following we 
occasionally omit the intersections corresponding to $\Lambda_i$ with $\vert\Lambda_i\vert = 1$.

Now let $\tilde{\Lambda}_i := \Lambda_i \cap \Delta_X$ for $i = 1, \ldots, p$. If $\tilde{\Lambda}_i = \emptyset$ for all $i$,
then $Y \subseteq X$ and hence $X + Y = X \in L(\CA)$.
Thus we may assume that $\{i \mid \tilde{\Lambda}_i \neq \emptyset\} = \{1, \ldots, q\}$ for some $1 \leq q \leq p$. Then we get
$$X + Y = \bigcap\limits_{i \in \Delta_X \cap \Delta_Y}H_i \cap \bigcap\limits_{i, j \in \tilde{\Lambda}_1}
H_{ij}(\theta_{ij}) \cap \cdots \cap \bigcap\limits_{i, j \in \tilde{\Lambda}_q}H_{ij}(\theta_{ij}) \in L(\CA),$$
and we are done.
\end{proof}
\end{lemma}

\begin{lemma}
\label{lem:startinduction}
Let $\CA = \CA^k_3(r)$ with $1 \leq k \leq 3$ and $r \geq 2$ and let $X \in L(\CA)$ of dimension $1$.
Then  $X$ is
modular in $L(\CA)$ if and only if $X = H_i \cap H_j$ for some $1 \leq i < j \leq k$.

\begin{proof}
By Lemma \ref{lem:modularelements} every such element is modular, so in the following we assume that $X$ is not of this form. If
$X = H_i \cap H$ for some $H \in \CA$ and $k < i \leq 3$, then a simple argument shows that $X$ is not modular. Thus we may
assume that $X$ is either $X^\eta = H_1 \cap H_{23}(\eta)$ or $X_{\theta_2 \theta_3} = H_{12}(\theta_2) \cap H_{13}(\theta_3)$.
A simple calculation shows that $X^\eta + X_{\theta_2 \theta_3} = \ker((\theta_2\eta + \theta_3)x_1 - \eta x_2 - x_3)$ which is
not contained in $L(\CA)$ whenever $\theta_2\eta + \theta_3 \neq 0$ and hence proves the lemma.
\end{proof}
\end{lemma}

\begin{proposition}
\label{prop:boolsub}
Let $\ell \geq 3$, $0 \leq k \leq \ell$ and $r \geq 2$ and let $\CA = \CA^k_\ell(r) \not\cong \CA^0_3(2)$. Then $X \in L(\CA)$
is modular in $L(\CA)$ if and only if $X \in L(\CB^k_\ell)$ or $X \in \CA \cup \{0, \BBC^\ell\}$.

\begin{proof}
Let $X \in L(\CA)$ with $1 \leq \dim X = p \leq \ell - 2$, that is $X$ is none of the obviously modular elements, cf.\ \cite[Ex.\ 2.28]{orlikterao:arrangements}.
Using the construction in \cite[\S 2]{orliksolomon:unitaryreflectiongroups} (\cite[\S 6.4]{orlikterao:arrangements}),
we can find a subset $\Delta_X \subseteq I$ and a
partition $\Lambda_X = (\Lambda_1, \ldots, \Lambda_p)$ of $I\setminus\Delta_X$ together with $r\th$ roots of unity $\theta_j$
such that
$$X = \bigcap\limits_{j \in \Delta_X}H_j \cap \bigcap\limits_{j \in \Lambda_1\setminus\{m_1\}}H_{m_1j}(\theta_j)\cap
\cdots \cap \bigcap\limits_{j \in \Lambda_p\setminus\{m_p\}}H_{m_pj}(\theta_j),$$
for $m_i := \min\Lambda_i$ ($1 \leq i \leq p$). We may assume that $\vert\Lambda_1\vert \geq \ldots \geq \vert\Lambda_p\vert$.
It is easy to see that $X$ is not modular whenever $\Delta_X \not\subseteq \{1, \ldots, k\}$, so we assume $\Delta_X \subseteq
\{1, \ldots, k\}$. Then we have $\vert\Lambda_1\vert = \ldots = \vert\Lambda_p\vert = 1$ if and only if $X \in L(\CB^k_\ell)$.
By Lemma \ref{lem:modularelements}, the elements of $L(\CB^k_\ell)$ are modular. Thus we assume $\vert\Lambda_1\vert > 1$ and
show that $X$ is not modular. For $H \in \CA$, we know by Proposition \ref{prop:intermediate} that $\CA^H \cong
\CA^{k'}_{\ell - 1}(r)$ with $0 < k' \leq \ell - 1$. Using Lemmas \ref{lem:modularrestrictions},
\ref{lem:startinduction} and induction on $\ell$, it suffices to show that for $\ell \geq 4$ we can choose $H \in \CA$ such that
$X \cap H$ is not modular in $\CA^{k'}_{\ell - 1}(r)$. Now we discuss two cases:

\paragraph{\textbf{Case 1:}} The rank of $X$ is at least $3$. Here we have two subcases:

\subparagraph{\textbf{Case 1a:}} 
If $\Delta_X \neq \emptyset$, then choose $i \in \Delta_X$ and set $H := H_i$. Considering Table
\ref{restriction-table}, we see
that in $L(\CA^H)$ the element $X \cap H = X$ is of the form 
$$X = \bigcap\limits_{j \in \Delta_X\setminus\{i\}}H_j \cap \bigcap\limits_{j \in \Lambda_1\setminus\{m_1\}}H_{m_1j}(\theta_j)
\cap \cdots \cap \bigcap\limits_{j \in \Lambda_p\setminus\{m_p\}}H_{m_pj}(\theta_j).$$
We still have $\vert\Lambda_1\vert > 1$ and hence $X$ is not contained in the lattice of the Boolean subarrangement
$\CB^{k'}_{\ell - 1}$ of $\CA^H$. As $X$ is clearly neither zero nor a hyperplane in $H$ (nor $H$ itself), it is not
modular in $L(\CA^H)$ by induction hypothesis.

\subparagraph{\textbf{Case 1b:}} If $\Delta_X = \emptyset$, then we may assume $\vert\Lambda_1\vert \geq 3$ or
$\vert\Lambda_1\vert = \vert\Lambda_2\vert = 2$,
because $X$ is not contained in $\CA$. In both cases we set $H := H_{m_1i}(\theta_i)$ for some $i \in \Lambda_1
\setminus\{m_1\}$. Now using Table \ref{restriction-table}, we see that $X = X \cap H \in L(\CA^H)$ is of the form
$$X = \bigcap\limits_{j \in \Lambda_1\setminus\{i, m_1\}}H_{ij}(\theta_i\inverse\theta_j) \cap \cdots \cap \bigcap\limits_{j
\in \Lambda_p\setminus\{m_p\}}H_{m_pj}(\theta_j),$$
where in the case $\vert\Lambda_1\vert = 2$ the first intersection is empty.
Clearly, we have $\vert\Lambda_1\setminus\{i\}\vert > 1$ or $\vert\Lambda_2\vert > 1$ and hence $X$ is not contained in the
lattice of the
Boolean subarrangement $\CB^{k'}_{\ell - 1}$ of $\CA^H$. As $X$ is clearly neither zero nor a hyperplane in $H$ (nor $H$ itself),
it is not modular in $L(\CA^H)$ by induction hypothesis.

\paragraph{\textbf{Case 2:}} Now let $X$ be of rank $2$. We may assume that $X$ is either $X_a = H_1 \cap H_{23}(\theta)$, $X_b
= H_{12}(\theta_2) \cap
H_{13}(\theta_3)$ or $X_c = H_{12}(\theta_2) \cap H_{34}(\theta_4)$.
We set $H := H_{24}(\eta) \in \CA$ and set $\tilde{X}_\gamma := X_\gamma \cap H$ for $\gamma \in \{a, b, c\}$.
Clearly, $\tilde{X}_\gamma \in L(\CA^H)$ ($\gamma = a, b, c$) is none of the trivially modular elements.
Table \ref{restriction-table} shows that in $L(\CA^H)$ we get the following cases
$$\tilde{X}_a = H_1 \cap H_{34}(\theta\inverse\eta), \quad \tilde{X}_b = H_{14}(\theta_2\eta) \cap H_{13}(\theta_3)
\quad\text{and}\quad \tilde{X}_c = H_{14}(\theta_2\eta) \cap H_{34}(\theta_4).$$
Thus $\tilde{X}_\gamma$ ($\gamma = a, b, c$) is not contained in the lattice of the Boolean subarrangement $\CB^{k'}_{\ell - 1}$
of $\CA^H$ and hence it is not modular by induction hypothesis.
\end{proof}
\end{proposition}

We obtain the following immediate consequence of Proposition \ref{prop:boolsub}.

\begin{corollary}
\label{cor:modelem1}
Let $\CA = \CA(G(r,p,\ell))$ and $X \in L(\CA)$. Then $\CA^X$ is supersolvable if and only if $L(\CA^X)$ contains a modular
element of dimension $1$.
\end{corollary}

\begin{proof}
It follows from Lemma \ref{lem:intermediates} and Proposition \ref{prop:boolsub} that  $\CA^k_\ell(r)$ is supersolvable
if and only if there is a modular element of dimension $1$ in $L(\CA^k_\ell(r))$. Now the proofs of
\cite[Prop.~2.5]{orliksolomon:unitaryreflectiongroups} and \cite[Prop.~2.14]{orliksolomon:unitaryreflectiongroups}
(see also \cite[Prop.~6.77 and 6.84]{orlikterao:arrangements}) imply the assertion.
\end{proof}

We record another consequence of  Proposition \ref{prop:boolsub} which together with \cite[Prop.~2.14]{orliksolomon:unitaryreflectiongroups}
(\cite[Prop.~6.84]{orlikterao:arrangements}) gives an alternative proof of Theorem \ref{super-restriction}(ii).

\begin{corollary}

$\CA^k_\ell(r)$ is supersolvable if and only if $k \in \{\ell - 1, \ell\}$ or $(k, \ell, r) = (0, 3, 2)$.

\end{corollary}

In the remainder of this section we 
consider the case when $W$ is of exceptional type.
We start with a further consequence of Lemma \ref{lem:dim3}.

\begin{corollary}
\label{cor:4dim1mod}
Let $\CA = \CA(W)$ be an irreducible, non-supersolvable  
reflection arrangement of exceptional type.
Let $X \in L(\CA)$ with $\dim X = 4$. 
Then $\CA^X$ is supersolvable if and only if 
$L(\CA^X)$ 
contains a modular element of dimension $1$.
\end{corollary}

\begin{proof}
Thanks to Lemma \ref{lem:dim3}, for every $W$ as in the statement, 
other than $W$ of type $E_6$, $E_7$ or $E_8$,
every $3$-dimensional restriction already fails to be supersolvable.
Consequently, using Lemma \ref{lem:super-rank2},
there are no modular elements of dimension $1$ in $L(\CA)$.

We are left to check the $4$-dimensional restrictions 
for  $W$ of type $E_6$, $E_7$ and $E_8$.

Using \cite[App.]{orliksolomon:unitaryreflectiongroups} (cf. \cite[App.\ C]{orlikterao:arrangements}), it is easily seen that 
each of 
$(E_7,(A_1^3)')$, $(E_8,A_1^4)$, $(E_8,A_1^2A_2)$, and $(E_8,D_4)$ has no 
supersolvable restriction to a hyperplane 
and therefore none of them admits a modular element of dimension $1$, 
by Lemmas \ref{lem:super-rank2} and \ref{cor:stanley}.

The next set of instances we consider are ones which 
have
only one type of restriction to a hyperplane which is supersolvable.
We can then use the condition of Lemma \ref{lem:show4nonmodular1} to show 
that none of these cases admits a modular element of dimension $1$.
We use the tables of \cite[App.]{orliksolomon:unitaryreflectiongroups}
(\cite[App.\ C]{orlikterao:arrangements}) to determine the number of restrictions
which are supersolvable and denote this number again by $c$, as in Lemma 
\ref{lem:show4nonmodular1}:\\
$(E_6,A_1^2)$: the supersolvable restriction has exponents $\{1,3,4\}$ and $c=3$.\\
$(E_7,(A_1^3)'')$: the supersolvable restriction has exponents $\{1,3,5\}$ and $c=1$.\\
$(E_7,(A_1A_2)')$: the supersolvable restriction has exponents $\{1,5,7\}$ and $c=4$.\\
$(E_7,A_3)$: the supersolvable restriction has exponents $\{1,3,5\}$and $c=3$.\\
$(E_8,A_2^2)$: the supersolvable restriction has exponents $\{1,5,7\}$ and $c=6$.\\
$(E_8,A_1A_3)$: the supersolvable restriction has exponents $\{1,5,7\}$ and $c=4$.\\

Unfortunately, Lemma \ref{lem:show4nonmodular1} does not apply for 
$(E_6,A_2)$ and $(E_8,A_4)$, since $c$ is too big in these 
instances. Using a computer, we checked directly 
that there are no modular elements of dimension $1$ in these two cases
(cf.\ Remark \ref{rem:computations}).
\end{proof}

\begin{corollary}
\label{cor:dim1mod}
Let $\CA = \CA(W)$ be an irreducible, non-supersolvable  
reflection arrangement of exceptional type.
Let $X \in L(\CA)$ with $\dim X \ge 3$. Then $\CA^X$ is supersolvable if and only if $L(\CA^X)$ contains a modular
element of dimension $1$.
\end{corollary}
\begin{proof}
For $\dim X = 3$, 
the assertion follows from 
Lemma \ref{lem:super-rank2} and Corollary \ref{cor:4dim1mod} shows that 
this also holds for $\dim X = 4$. Since all restrictions $\CA^X$ with $\dim X\ge 4$
are not supersolvable for the exceptional types, the result follows for $\dim X>4$
by Lemma \ref{lem:modularrestrictions}.
\end{proof}

\begin{proof}[Proof of Theorem \ref{super-restriction-rank2}]
The forward implication is clear from Definition \ref{def:super}.

The reverse implication follows 
for general central arrangements of rank up to $3$
by \cite[Rem.\ 2.3]{hogeroehrle:super} and Lemma \ref{lem:super-rank2}.
For $W = G(r, p, \ell)$, 
the result is  Corollary \ref{cor:modelem1}.
Finally, for $W$ of exceptional type and 
$X \in L(\CA)$ with $\dim X \geq 4$
such that $\CA^X$ is not supersolvable,
the result follows from 
Corollary \ref{cor:dim1mod}.

Note that 
Corollary \ref{cor:modelem1} and
Corollary \ref{cor:dim1mod} also 
apply in the case $X = V$.
This gives the equivalence in the statement for the underlying  
reflection arrangement $\CA(W)$ itself 
(cf.\ Remark \ref{rem:super-restriction-rank2}(i)).
\end{proof}

\begin{remark}
\label{rem:computations}
In order 
to establish several of our results 
we first use the functionality for complex reflection groups 
provided by the   \CHEVIE\ package in   \GAP\ 
(and some \GAP\ code by J.~Michel)
(see \cite{gap3} and \cite{chevie})
in order to obtain explicit 
linear functionals $\alpha$ defining the hyperplanes 
$H = \ker \alpha$ of the reflection arrangement
$\CA(W)$ and the relevant restrictions $\CA(W)^X$.
 
We then use the functionality of
\Sage\  (\cite{sage}) to determine the data in 
Table \ref{table:rank3restrictions} and in the proof of Lemma \ref{lem:dim3}.
Moreover, \Sage\ was used
to establish the fact that 
the two restrictions 
$(E_6, A_2)$ and $(E_8, A_4)$
do not admit modular elements of dimension $1$
in the proof of Corollary \ref{cor:4dim1mod}.
\end{remark}


\bigskip {\bf Acknowledgments}: 
We acknowledge 
support from the DFG-priority program 
SPP1489 ``Algorithmic and Experimental Methods in
Algebra, Geometry, and Number Theory''.


\bigskip

\bibliographystyle{amsalpha}

\newcommand{\etalchar}[1]{$^{#1}$}
\providecommand{\bysame}{\leavevmode\hbox to3em{\hrulefill}\thinspace}
\providecommand{\MR}{\relax\ifhmode\unskip\space\fi MR }
\providecommand{\MRhref}[2]{%
  \href{http://www.ams.org/mathscinet-getitem?mr=#1}{#2} }
\providecommand{\href}[2]{#2}


\end{document}